%% file: oo-main.tex
\documentclass[reqno,12pt]{amsart} 

\usepackage{graphicx}
\usepackage{hyperref} %
\usepackage{amssymb,amsmath,amstext}
\usepackage{latexsym}
\usepackage{microtype} 
\usepackage{algorithmic}
\usepackage{xcolor}

\newtheorem{theorem}{Theorem}
\newtheorem{lemma}{Lemma}

\begin{document}

\markboth{J.J. Brust, J.B. Erway, and R.F. Marcia}{Shape-Changing Trust-Region Methods Using Multipoint Symmetric Secant Matrices}

\title[Shape-Changing Trust-Region Methods Using Multipoint Symmetric Secant Matrices]{Shape-Changing Trust-Region Methods Using Multipoint Symmetric Secant Matrices}

\author[J.J. Brust]{Johannes J. Brust}
\email{jjbrust@ucsd.edu}
\address{Department of Mathematics, University of California, San
  Deigo, La Jolla, CA  92093}

\author[J.B. Erway]{Jennifer B. Erway}
\email{erwayjb@wfu.edu}
\address{Department of Mathematics, Wake Forest University, Winston-Salem, NC 2\
  7109}

\author[R.F. Marcia]{Roummel F. Marcia}
\email{rmarcia@ucmerced.edu}
\address{School of Natural Sciences, University of California, Merced, Merced,
  CA 95343}

\thanks{R.~F. Marcia is supported in part by National Science Foundation grant
IIS-1741490.}

\begin{abstract}
In this work, we consider methods for large-scale and nonconvex unconstrained optimization. We propose a new trust-region method whose 
subproblem is defined using a so-called ``shape-changing" norm together with densely-initialized multipoint symmetric secant ({\small MSS}) matrices to approximate the Hessian.  Shape-changing norms and dense 
initializations have been successfully used in the context
of traditional quasi-Newton methods, but have yet to be explored in the case
of {\small MSS} methods.  Numerical results suggest that
trust-region methods that use densely-initialized
{\small MSS} matrices together with shape-changing norms outperform {\small MSS} with other trust-region methods. 
\end{abstract}


\maketitle
\makeatletter
\newcommand{\defined}{\mathop{\,{\scriptstyle\stackrel{\triangle}{=}}}\,}

\newcommand{\rfm}[1]{\textcolor{black}{#1}}
\newcommand{\obs}[1]{\textcolor{black}{#1}}
\newcommand{\je}[1]{\textcolor{black}{#1}}
\newcommand{\jeo}[1]{\textcolor{black}{#1}}
\newcommand{\jb}[1]{\textcolor{black}{#1}}
\newcommand{\jbo}[1]{\textcolor{black}{#1}}

\newcommand{\bp}{\mathbf{p}}
\newcommand{\bv}{\mathbf{v}}
\newcommand{\bfm}[1]{\mathbf{#1}} 
\newcommand{\bk}[1]{\mathbf{#1}_k} 
\newcommand{\bko}[1]{\mathbf{#1}_{k+1}} 
\newcommand{\bsk}[1]{\boldsymbol{#1}_k} 
\newcommand{\bsko}[1]{\boldsymbol{#1}_{k+1}} 
\newcommand{\bs}[1]{\boldsymbol{#1}} 

\newcounter{pseudocode}[section]
\def\thepseudocode{\thesection.\arabic{pseudocode}}
\newenvironment{pseudocode}[2]%
        {%
        \refstepcounter{pseudocode}%
          \AlgBegin %
               {{\bfseries Algorithm \thepseudocode.}\rule[-1.25pt]{0pt}{10pt}#1}%
        #2}%
           {\AlgEnd}

\newcounter{Pseudocode}[section]
\def\thePseudocode{\thesection.\arabic{Pseudocode}}
\newenvironment{Pseudocode}[2]%
        {%
        \refstepcounter{Pseudocode}%
          \AlgBegin %
               {{\bfseries #1.}\rule[-1.25pt]{0pt}{10pt}}%
        #2}%
           {\AlgEnd}

\def\tnu{\tilde{\nu}}

\newcounter{procedureC}
\newcounter{algorithm saved}
\newenvironment{procedureAlg}[1][H]{%
	\setcounter{algorithm saved}{\value{algocf}} 
    	\setcounter{algocf}{\value{procedureC}}
	\renewcommand{\algorithmcfname}{PROCEDURE}
   \begin{algorithm}[#1]%
  }{\end{algorithm}
  \setcounter{procedureC}{\value{algocf}}
  \setcounter{algocf}{\value{algorithm saved}} 
  }

\makeatother

\input{intro.tex}

\input{background.tex}

\input{implementation.tex}
\input{numerical.tex}

\input{conclusion.tex}

\section*{Acknowledgments}
This research work was partially funded by NSF Grant IIS-1741490.

\newpage
\bibliographystyle{unsrt}
\bibliography{references}

\end{document}

%% file: intro.tex
\section{Introduction}
In this paper we propose a new solver for general large nonconvex problems of the following form: \begin{equation}\label{eqn-min}\min f(x),\end{equation}
where $x\in\Re^n$ and $f$ is continuously differentiable.
Generally speaking, solvers for unconstrained optimization  fall into two categories: Line search and trust-region methods.  While traditional line search methods that use local quadratic models require each of these models are convex, trust-region methods are able to approximate nonconvexity in the underlying function using nonconvex quadratic models. 
In this work, we focus on trust-region methods with possibly-indefinite Hessian approximations.  

Trust-region methods
 generate a sequence of iterates $\{x_k\}$ by
solving 
at each iteration a \emph{trust-region} subproblem:
\begin{equation}\label{eqn-trsp}
\min_{s\in\Re^n} Q(s)=g_k^Ts+\frac{1}{2}s^TB_ks
\quad \text{subject to } \quad \|s\|\le \delta,
\end{equation}
where $g_k=\nabla f(x_k)$ and $B_k\approx \nabla^2 f(x_k)$.  
Given $x_k$, the next iterate is then computed as $x_{k+1}=x_k+s_k$,
where $s_k$ is an  approximate solution to (\ref{eqn-trsp}).  
The choice of norm in (\ref{eqn-trsp}) affects the
difficulty to solve the trust-region
subproblem. In this manuscript, we consider two \emph{shape-changing}~\cite{BYuan02} norms to
define the trust region.  By using one of the norms, the trust-region subproblem decouples into two subproblems, each with closed-form solutions.  Meanwhile, the other shape-changing norm allows for a similar decoupling into two subproblems, one with a closed-form solution and the other that is a low-dimensional trust-region subproblem that is easily solved.

\medskip

In large-scale optimization, it can be the case that
the Hessian of the objective function is either too computationally expensive
to compute or store.  In these cases, 
first-order methods such as steepest descent and quasi-Newton methods may give the fastest convergence.
Multipoint symmetric secant ({\small MSS}) methods
can be thought of as generalizations of quasi-Newton methods in that they attempt to enforce multiple secant conditions at one time.
As with quasi-Newton methods,  {\small MSS} methods
generate a sequence of matrices $\{B_k\}$
to approximate the Hessian of $f$ at $x_k$ using
a sequence of low-rank updates.  Specifically, at
each iteration, $B_k$ is updated using a recursion relation where
$B_{k+1}=B_k+U_k$, and $U_k$ is a low-rank update (e.g., $\text{rank}(U_k)\le 2$).  However, these methods must
be user-initiated by selecting an initial $B_0$.
Conventionally, the initial matrix for quasi-Newton matrices is chosen to
be a scalar multiple of the identity, e.g., $B_0=\gamma_k I$, $\gamma_k\in\Re$.  This choice
leads to minimal storage requirements; only $\gamma_k$ 
must be stored in addition to the quasi-Newton pairs.
More recently, dense initializations have been proposed that
are also low-memory initializations in which only two constants must be stored in addition to the usual quasi-Newton pairs~\cite{denseInit,mss-erway}.  These dense
initializations implicitly decompose $\Re^n$ into
two orthogonal subspaces, assigning a parameter to
each subspace.  In this paper, we consider the same splitting of $\Re^n$ in a shape-changing trust-region setting. 

\medskip

In this manuscript, we propose a
densely-initialized {\small MSS} method together with
a shape-changing trust region to solve large nonconvex
problems of the form (\ref{eqn-min}).
The motivation for this
research comes from three different recent papers on first-order methods for solving general nonconvex unconstrained optimization problems where (1) a limited-memory symmetric rank-one 
({\small L-SR1}) quasi-Newton trust-region method with a shape-changing norm was found to be competitive with other standard trust-region methods~\cite{Oleg2021}, 
(2) a limited-memory
Broyden-Fletcher-Goldfab-Shanno
({\small L-BFGS}) method with a dense initialization together with a shape-changing trust-region method outperformed other {\small L-BFGS} trust-region methods~\cite{denseInit} and (3) a
{\small MSS} trust-region method that uses the Euclidean norm to define the trust region together with
a dense initialization outperformed both conventional initalizations and other standard quasi-Newton trust-region methods~\cite{mss-erway}. 

\medskip

This paper is organized in five sections.  Background on {\small MSS} matrices, shape-changing norms, and the dense initialization is given in Section 2.  In Section 3, we present the contributions of this research; namely, we propose  densely-initialized {\small MSS} trust-region methods that use the  shape-changing norms.  Numerical results on the {\small CUTE}st test set are presented in Section 4 comparing
the proposed method to other quasi-Newton methods.  Finally, concluding remarks are in Section 5.

\subsection{Notation and Glossary}
Throughout this paper, capital letters denote matrices and lower-case letter are reserved for vectors.
Moreover, vectors with an asterisk denote optimal
solutions.
The symbol $e_i$ denotes
the $i$th canonical basis vector whose dimension depends on context. Finally, ``sgn" denotes the signum function.

\subsection{Dedication}
We dedicate this paper to Oleg P. Burdakov.
The work presented here synthesizes
 three of Oleg's many areas of research: shape-changing norms~\cite{BYuan02,Oleg2021,burdakov2017efficiently}, dense initalizations of quasi-Newton methods~\cite{denseInit}, and {\small MSS} and secant  methods~\cite{burdakov1983methods,burdakov83stable,burdakov86,burdakov2002limited}.  
This manuscript is written in his memory.

%% file: background.tex
\section{Background}\label{sec-back}
In this section, we review {\small MSS} matrices, including
a recursion relation and compact formulation,
and the shape-changing norm.

\subsection{MSS matrices}
{\small MSS} matrices are generated similarly to traditional quasi-Newton matrices: A sequence
of matrices $\{B_k\}$ is recursively formed using a sequence of
low-rank updates.  
Specifically, if $\{x_k\}$ is a sequence
of updates obtained to solve (\ref{eqn-min}), then define
$s_k=x_{k+1}-x_k$ and $y_k=\nabla f(x_{k+1})-f(x_k)$.
The pairs $\{(s_k,y_k)\}$ are often referred to as \emph{quasi-Newton} pairs.  In methods for large-scale optimization, \emph{limited-memory} versions of these methods are used that store only the most recently-computed $m$ pairs; $m$ is typically referred to as the ``memory'' of the method.  In practice, $m\ll n$, e.g., $m\in[3,7]$ (see, e.g., \cite{Representations}).  In this work, we assume a limited-memory framework where $m$ is small; further, we let $l$ denote the current number of stored pairs.

Let the matrices $S_k$ and $Y_k$ denote the matrices whose columns are formed by the stored quasi-Newton pairs:
\begin{equation}\label{eqn-SY}
S_k=[s_{k-1} \,\, s_{k-2} \,\, \ldots \,\, s_{k-l}]\in\Re^{n\times l}\quad\text{and} \quad
Y_k=[y_{k-1} \,\, y_{k-2} \,\, \ldots \,\, y_{k-l}]\in\Re^{n\times l},
\end{equation}
where $l\le m$. 
While quasi-Newton matrices must satisfy the so-called \emph{secant condition} $B_{k+1}s_k=y_k$,
{\small MSS} matrices seek to satisfy 
multiple secant conditions:  $B_kS_k=Y_k$.  Generally speaking,
it is impossible to satisfy the secant conditions and also require $B_k$ to be symmetric~\cite{schnabel1983quasi}.  (This can be seen by multiplying the multiple secant equations by $S_k^T$ on the left and noticing that while $S_k^TB_kS_k$ is symmetric whenever $B_k=B_k^T$, it is not generally true that $S_k^TY_k$ will be symmetric~\cite{schnabel1983quasi}.)

In~\cite{burdakov1983methods,burdakov2002limited,burdakov91}, Burdakov proposes relaxing the secant conditions by symmetrizing the product $S_k^TY_k$ using the following symmetrization:
$$\text{sym}(A)=\left\{
\begin{array}{l}
  A_{ij}, \,\, i\ge j \\
   A_{ji}, \,\, i<j. 
\end{array} \right. $$
With this symmetrization, $S_k^TB_kS_k=\text{sym}(S_k^TY_k)$ and yields the following  recursion relation for 
$B_k$:  \begin{equation}\label{eqn-rank2}
B_{k+1}=B_k+\frac{(y_k-B_ks_k)c_k^T+c_k(y_k-B_ks_k)^T}{s_k^Tc_k}
-\frac{(y_k-B_ks_k)^Ts_kc_kc_k^T}{(s_k^Tc_k)^2},\end{equation}
where $c_k\in\Re^n$ is any vector such that $c_k^Ts_i=0$ for all $0\le i <k$
and $c_k^Ts_k\ne 0$~\cite{burdakov1983methods,brust2018large}. Notice that this is a rank-two update;
in practice, this update can generate indefinite approximations to the Hessian.
More generally, the recursion relation (\ref{eqn-rank2}), without the orthogonality conditions on $c_k$, is not new in the
literature:  When $c_k=s_k$, the update is the Powell-symmetric-Broyden ({\small PSB}),
and $c_k=y_k$ yields the Davidon-Fletcher-Powell {\small DFP} update.

Given any initial $B_0$, the general compact formulation for {\small MSS} matrices is  $B_k=B_0+\Psi_kM_k\Psi_k^T$, where
\begin{equation}\label{eqn-cptB}
  \Psi_k\triangleq \begin{bmatrix} S_k & (Y_k-B_0S_k)\end{bmatrix} \quad \text{and}
  \quad
  M_k\triangleq \begin{bmatrix} W(S_k^TB_0S_k-(T_k+E_k+T_k^T))W & W \\ W & 0\end{bmatrix},
\end{equation}
where $W=(S_k^TS_k)^{-1}$, $T_k$ is the
strict upper triangular portion of $S_k^TY_k$, and $E_k$ is the diagonal of $S_k^TY_k$~\cite{brust2018large}.
The compact formulation requires that $S_k$ has full rank.  One way this can be accomplished is using a rank-revealing decomposition and then removing columns of $S_k$ that are linearly dependent.   In~\cite{mss-erway}, the $LDL^T$ decomposition of $S_k^TS_k$ is used to find linear dependence.  In this case, if a column of $S_k$ is removed then the corresponding column of $Y_k$ must also be removed in order
for $M_k$ to be well-defined.  See~\cite{mss-erway} for a full discussion.

\subsection{The spectral decomposition}\label{sec-spectral}
Consider the compact
formulation of $B_k$, with
(\ref{eqn-cptB}):
\begin{equation*}
B_k=B_0+\Psi_kM_k\Psi_k^T,
\end{equation*}
where $\Psi_k\in\Re^{n\times 2l}$, $M_k\in\Re^{2l\times 2l}$, and $l$ is the number of stored quasi-Newton
pairs.  Then, given the so-called ``thin" {\small QR} factorization of $\Psi_k$, namely $\Psi_k=QR$, 
we obtain the following expression for $B_k$:
\begin{equation}\label{eqn-spectral}
B_k=B_0+QRM_kR^TQ^T,
\end{equation}
where $Q\in\Re^{n\times n}, R\in\Re^{l\times l}$, and $RM_kR^T\in\Re^{2l\times 2l}$.  Because of
its small size, the spectral decomposition of 
$RM_kR^T$ is computable.  Suppose $U\hat{\Lambda}U^T$
is the spectral decomposition of $RM_kR^T$ with
$\hat{\Lambda}=\text{diag}\{\hat{\lambda}_1,\ldots,\hat{\lambda}_{2l}\}$, and
$B_0=\gamma_k I$ is the conventional one-parameter
initialization.
Then, $B_k$ can be written as
$B_k=P\Lambda P^T$, where
\begin{equation*}P= \begin{bmatrix} QU & (QU)^\perp\end{bmatrix} \quad \text{and} \quad
\Lambda = \begin{bmatrix} \hat{\Lambda}+\gamma_k I & 0 \\ 0 & \gamma_kI \end{bmatrix}.
\end{equation*}
While the above derivation is found in~\cite{denseInit,Oleg2021,burdakov2017efficiently,simaxErway,brust2017solving}, in practice, 
the factorization may also be accomplished using the $LDL^T$ factorization~\cite{burdakov2017efficiently,denseInit,Oleg2021,mss-erway}.  (For
more details on this spectral decomposition of $B_k$
see~\cite{burdakov2017efficiently,denseInit,Oleg2021,mss-erway}.)  For simplicity, we define
\begin{equation}\label{eqn-P}
P_\parallel=QU \quad \text{and} \quad P_\perp=(QU)^\perp,
\end{equation}
and make use of these definitions throughout the duration of the manuscript.

\medskip

In order to solve the trust-region subproblem at each iteration, it will be necessary to be able to implicitly perform matrix-vector products with $P_\parallel$.  It is possible to form $P_\parallel$ without storing $Q$.  To see this, note that 
\begin{equation}\label{eqn-Parallel}
P_\parallel=QU=\Psi_kR^{-1}U.
\end{equation} In order for $R$ to be invertible, $\Psi_k$ must have full rank. Similar to~\cite{burdakov2017efficiently,Oleg2021,mss-erway,denseInit}, we propose using the $LDL^T$ decomposition to identify columns of $\Psi_k$ that are linearly dependent.  For a full discussion on ensuring both $S_k$ and $\Psi_k$ have linearly independent columns, see~\cite{mss-erway}.

\subsection{Shape-changing norms}
Trust-region subproblems have the form of (\ref{eqn-trsp}),
 but any norm may be used in the constraint that defines the trust region.  The most commonly-chosen norm is the Euclidean norm; other popular choices found in the literature are the one-norm and infinity-norm.  One important advantage in using the Euclidean norm to define the trust region is that there are optimality conditions that characterize a global solution~\cite{Gay81,more1983computing}:
\begin{theorem}\label{thrm-GMS} The vector
$s^*\in\Re^n$ is a global
solution of
\begin{equation}\label{eqn-trsp2}
\min_{s\in\Re^n} Q(s)=g_k^Ts+\frac{1}{2}s^TBs
\quad \text{subject to } \quad \|s\|_2\le \delta,
\end{equation}
if and only if $\|s^*\|_2\le \delta$
and there exists a unique $\sigma^*\ge 0$ such that $B+\sigma^* I$ is
positive semidefinite and
\begin{equation}\label{eqn-opt}
  (B+\sigma^*I)s^*=-g \quad \text{and} \quad
\sigma^*(\delta-\|s^*\|_2)=0.
\end{equation}
\end{theorem}
So-called ``exact" subproblems solvers
aim to explicitly find a pair $(s^*,\sigma^*)$
that satisfy (\ref{eqn-opt}) in order to solve each subproblem to high accuracy~\cite{more1983computing,brust2017solving,mss-algo,burdakov2017efficiently}.  While convergence proofs of the overall trust-region method only require an approximate solution
of each subproblem~\cite{powell40,powell41,powell42},  these ``exact" solvers bet on that solving each subproblem to high precision will lead to fewer overall iterations of the trust-region method.

\medskip

The ``shape-changing" norms were first defined by Burdakov and Yuan~\cite{BYuan02}.
These shape-changing norms make use of the matrix
of eigenvectors $P_\parallel$ and $P_\perp$ (see \ref{eqn-P}), and thus, the size and shape of the
trust-region changes every iteration:
\begin{eqnarray}
\|s\|_{P,\infty} & = &  \max\left(\|P_\parallel^Ts\|_\infty,
\|P_\perp^Ts\|_2\right) \label{eqn-shape1}\\
\|s\|_{P,2} & = &  \max\left(\|P_\parallel^Ts\|_2,
\|P_\perp^Ts\|_2\right). \label{eqn-shape2}
\end{eqnarray}
For simplicity, we refer to (\ref{eqn-shape1}) as the $(P,\infty)$ norm,
and (\ref{eqn-shape2}) as the $(P,2)$ norm.

In~\cite{burdakov2017efficiently}, Burdakov et al.
show that these norms are equivalent to the Euclidean
norm and the equivalence factors are independent of $P$.
Importantly, these norms allow each subproblem to be
decomposed into two small subproblems, each of which
are either easy to solve or have a closed-form solution.
These norms have been successfully used in  {\small L-BFGS} and {\small L-SR1} trust-region settings~\cite{burdakov2017efficiently, Oleg2021}.  For more details on these norms, see~\cite{BYuan02,burdakov2017efficiently}.
Implementation details for solving  trust-region subproblems defined using the shape-changing norm are presented in Section~\ref{sec-trsp-shape}.

\subsection{The dense initialization}
The conventional initialization for a quasi-Newton method is a constant diagonal initialization, i.e.,
$B_0=\gamma_k I$, $\gamma_k\in\Re$.  This initialization performs well in practice and enjoys ease of use with no additional memory requirements other than storing a scalar--for these reasons it is the most popular initialization.  Other low-memory initializations include nonconstant diagonal matrices.
Until the dense initalization was first proposed for {\small L-BFGS} matrices, low-memory
intializations were limited to diagonal matrices.

\medskip

The dense initialization exploits the partitioning
of $\Re^n$ into two subspaces: (i) the eigenspace
associated with the eigenvalues $$\hat{\lambda}_1+\gamma_k ,\ldots,\hat{\lambda}_{2l}+\gamma_k,$$
and (ii) the eigenspace associated 
eigenvalue $\gamma_k$, assuming that $RM_kR^T$ is nonsingular.
Specifically, note that
$$B_0=\gamma_kI =
\gamma_k PP^T = 
\gamma_k P_\parallel P_\parallel^T + \gamma_k P_\perp P_\perp^T.$$
In lieu of using one parameter for both spaces,
the dense initalization uses two:
\begin{equation}\label{eqn-B0}
\tilde{B}_0=\zeta_k P_\parallel P_\parallel^T + \zeta^C_k P_\perp P_\perp^T,
\end{equation}
where $\zeta_k,\zeta_k^C\in\Re$.  For the duration
of the paper, $\tilde{B}_0$ will denote a dense initial matrix.  Using the dense initialization, the compact formulation becomes
$B_k=P\Lambda P^T$, where
\begin{equation}
\label{eqn-dense-spectral}
P= \begin{bmatrix} QU & (QU)^\perp\end{bmatrix} \quad \text{and} \quad
\Lambda = \begin{bmatrix} \hat{\Lambda}+\zeta_k I & 0 \\ 0 & \zeta_k^CI \end{bmatrix}.
\end{equation}

%% file: implementation.tex
\section{Implementation}\label{sec-implement}
In this section, we demonstrate how the dense initialization for an {\small MSS} method can be used in a shape-changing norm.  This section presents the contributions of this research.

\subsection{A second compact formulation}
\label{subsec-compact2}
The spectral decomposition $B_k=P\Lambda P^T$ where
$P$ and $\Lambda$ are given by (\ref{eqn-dense-spectral}) relies only on the
existence of a compact formulation for $B_k$.
In this subsection, we derive an alternative 
compact formulation for $B_k$ in the case
of a dense initialization that is compatible
with the derivation of the spectral decomposition
in Section~\ref{sec-spectral}.

\medskip

Consider the compact formulation for the
dense initialization: $$B_k=\tilde{B}_0+
\Psi_kM_k\Psi_k^T,$$
where $\Psi_k$ and $M_k$ are given in (\ref{eqn-cptB}).
The following lemma appears in~\cite{mss-erway}
and allows us to define an alternative compact
formulation in Theorem \ref{thrm-cpt2}:
\begin{lemma}\label{cor1}
     Suppose $\tilde{B}_0$ is as in (\ref{eqn-B0}).  Then, $\tilde{B}_0S_k=\zeta_k S_k$.
\end{lemma}\label{cor-S}
\begin{proof}
See~\cite[Corollary 3.4]{mss-erway}.
\end{proof}

\begin{theorem}\label{thrm-cpt2}
Suppose $\tilde{B}_0$ is as in (\ref{eqn-B0}) and
$S_k$ is full rank,
then $B_k$ can be written as 
$$B_k=\tilde{B}_0+
\tilde{\Psi}_k\tilde{M}_k\tilde{\Psi}_k^T,$$
where
\begin{equation}\label{eqn-hats}
\tilde{\Psi}_k= \begin{bmatrix} S_k & Y_k\end{bmatrix}
\quad \text{and}\quad
\tilde{M}_k=
   \begin{bmatrix} -\zeta_k W -W(T_k+E_k+T_k^T)W & W \\ W & 0\end{bmatrix},
\end{equation}
and $W=(S_k^TS_k)^{-1}$, 
$T_k$ is the
strict upper triangular portion of $S_k^TY_k$, and $E_k$ is the diagonal of $S_k^TY_k$.
\end{theorem}
\begin{proof}
Consider the compact formulation where 
$\Psi_k$ and $M_k$ are given
by (\ref{eqn-cptB}) together 
with the dense initialization:
$$B_k=\tilde{B}_0+\Psi_k M_k \Psi_k^T,$$
where
\begin{equation*}
  \Psi_k= \begin{bmatrix} S_k & (Y_k-\tilde{B}_0S_k)\end{bmatrix} \quad \text{and}
  \quad
  M_k= \begin{bmatrix} W(S_k^T\tilde{B}_0S_k-(T_k+E_k+T_k^T))W & W \\ W & 0\end{bmatrix}.
\end{equation*}

\noindent
By Lemma~\ref{cor1},  $\Psi_k=[S_k \,\,\, (Y_k-\zeta_kS_k)]$ and $M_k$ can be simplified as follows:

\begin{eqnarray*}
 M_k 
   & = &  \begin{bmatrix} W(S_k^T\tilde{B}_0S_k-(T_k+E_k+T_k^T))W & W \\ W & 0\end{bmatrix}
  \\
  &=&
   \begin{bmatrix} \zeta_k W -W(T_k+E_k+T_k^T)W & W \\ W & 0\end{bmatrix}.
 \end{eqnarray*}
Putting this together yields that $\Psi_kM_k\Psi_k^T$
can be written as
\begin{eqnarray*}
\Psi_kM_k\Psi_k^T
 &=& \begin{bmatrix} S_k & Y_k-\zeta_k S_k\end{bmatrix}
   \begin{bmatrix} \zeta_k W -W(T_k+E_k+T_k^T)W & W \\ W & 0\end{bmatrix}
  \begin{bmatrix}
  S_k^T \\
  (Y_k - \zeta_k S_k)^T 
  \end{bmatrix}
  \\
  & = & 
  -\zeta_k S_kWS_k^T - S_kW(T_k+E_k+T_k^T)WS_k^T 
  + S_kWY_k^T + Y_kWS_k^T   \\[.2cm]
    &=&
   \begin{bmatrix} S_k & Y_k\end{bmatrix}
   \begin{bmatrix} -\zeta_k W -W(T_k+E_k+T_k^T)W & W \\ W & 0\end{bmatrix}
  \begin{bmatrix}
  S_k^T \\
  Y_k^T 
  \end{bmatrix}.
\end{eqnarray*}
Thus, $B_k=\tilde{B}_0+\hat{\Psi}_k\hat{M}_k\hat{\Psi}_k^T$,
where $\tilde{\Psi}_k$ and $\tilde{M}_k$ are given by
(\ref{eqn-hats}), respectively.
\end{proof}
It is the case that the compact formulation given in Theorem~\ref{thrm-cpt2} is the compact
formulation Burdakov in~\cite{burdakov2002limited} derived for the case of the conventional initialization $B_0=\gamma_k I$, $\gamma_k\in I$; however, his method of derivation required a single-parameter initialization.

\medskip

There are several benefits of using the compact formulation given in Theorem~\ref{thrm-cpt2} over 
the compact formulation defined by (\ref{eqn-cptB}).  Namely,
$\Psi_k$ can be formed without
any computations; whereas (\ref{eqn-cptB}) requires
scalar multiplication with $\zeta_k$ and $n$ subtractions.
Moreover, more importantly,
whenever a new $\zeta_k$ is computed $\Psi_k$
in (\ref{eqn-cptB}) must
be recomputed, possibly from scratch; however, 
in (\ref{eqn-hats}), $\zeta_k$ is not needed
to form $\tilde{\Psi}_k$.  It is for these reasons
that the proposed method uses this second alternative compact formulation.

We note further that $B_k$ in \eqref{eqn-hats} can also be represented by  
\begin{equation*}
\widehat{\Psi}_k= \begin{bmatrix} S_k W & Y_k\end{bmatrix}
\quad \text{and}\quad
\widehat{M}_k=
   \begin{bmatrix} -\zeta_k W^{-1} -(T_k+E_k+T_k^T) & I \\ I & 0\end{bmatrix}.
\end{equation*} Since $ W^{-1} = S_k^T S_k $,  $ \widehat{M}_k $ can be formed without inverting a small matrix. In order to keep computational cost low, it is not necessary to form the product $ S_k W $. Instead, when a product with an arbitrary vector $p$ and $ \widehat{\Psi}_k $ is needed, one can compute this using only matrix-vector products as follows: 
$ \widehat{\Psi}_k^T p = \begin{bmatrix} W (S_k^T p) \\ Y_k^T p \end{bmatrix} $. We make this representation available as an additional option for an {L-MSS} method; however, our numerical experiments favored the results with \eqref{eqn-hats}.  For this reason, we assume the representation~\eqref{eqn-hats} for the duration of the paper.


\subsection{Solving the trust-region subproblem}
\label{sec-trsp-shape}
In this section, we demonstrate how to solve the trust-region subproblem defined by a shape-changing
norm, where an {\small MSS} matrix is used to approximate the Hessian at each iterate.  We generally follow the presentations in~\cite{burdakov2017efficiently} and~\cite{Oleg2021}, altering the presentation to allow for a dense initialization.

\subsubsection{The $(P,\infty)$ shape-changing norm}\label{sec-Pinfty}
In this section, we consider a trust-region subproblem whose constraint is defined by the $(P,\infty)$ norm:
\begin{equation}\label{eqn-shape-tr1}
\min_{s\in\Re^n} Q(s)=g_k^Ts+\frac{1}{2}s^TB_ks
\quad \text{subject to } \quad \|s\|_{P,\infty}\le \delta,
\end{equation}

\medskip
Consider $v=P^Ts$ where $P=[P_\parallel \,\, P_\perp]$ as in Section~\ref{sec-spectral}.  Applying this change of variables by substituting in $Pv$ for $s$ in (\ref{eqn-shape-tr1}) yields the following quadratic function:
\begin{equation}\label{eqn-Pv}
Q(Pv)=g_k^T(Pv)+\frac{1}{2}(Pv)^TB_k(Pv).
\end{equation}
Letting $$v_\parallel=P_\parallel^T s, \quad v_\perp = P_\perp^Ts, \quad g_\parallel=P_\parallel^Tg, \quad g_\perp=P_\perp^Tg,$$ then (\ref{eqn-Pv}) simplifies
as follows:
\begin{eqnarray}\label{eqn-cov1}
Q(Pv)& = & g_k^T(Pv)+\frac{1}{2}v^T\Lambda v \nonumber \\
  & = & g_\parallel^Tv_\parallel+g_\perp^Tv_\perp + \frac{1}{2}\left(v_\parallel^T\left(\hat{\Lambda}+\zeta_kI_{2l}\right)v_\parallel +\zeta^C_k \|v_\perp\|_2^2\right).
\end{eqnarray}

Notice that (\ref{eqn-cov1}) is
separable, and thus,(\ref{eqn-shape-tr1}) can be decoupled into two trust-region subproblems:
\begin{eqnarray}
\min_{v_\parallel\in\Re^{2l}} q_\parallel(v_\parallel)=g_\parallel^Tv_\parallel+\frac{1}{2}\left(v_\parallel^T\left(\hat{\Lambda}+\zeta_kI_{2l}\right)v_\parallel\right)\quad \text{subject to } \quad \|v_\parallel\|_{\infty}\le \delta, \label{eqn-shape-tr1-sep}\\
\min_{v_\perp\in\Re^{n-2l}}q_\perp(v_\perp)=g_\perp^Tv_\perp+\frac{1}{2}\zeta^C_k\|v_\perp\|_2^2
\quad \text{subject to } \quad \|v_\perp\|_{2}\le \delta.
\label{eqn-shape-tr2-sep}
\end{eqnarray}
Both of these subproblems have closed form solutions.
In particular,
the  closed-form solution to (\ref{eqn-shape-tr1-sep}) is given by~\cite{burdakov2017efficiently,Oleg2021}:
\begin{equation*}
\left[ v_\parallel^*\right]_i=
\begin{cases}
-\frac{[g_\parallel]_i}{\lambda_i} & \text{if }
\left| \frac{[g_\parallel]_i}{\lambda_i} \right|\le
\delta_k \text{ and } \lambda_i>0, \\
c & \text{if } [g_\parallel]_i=0 \text{ and }
\lambda_i=0, \\
-\text{sgn}\left( [g_\parallel]_i\right)\delta_k &
 \text{if } [g_\parallel]_i \ne 0 \text{ and }
\lambda_i=0, \\
\pm \delta_k & 
 \text{if } [g_\parallel]_i = 0 \text{ and }
\lambda_i<0, \\
-\frac{\delta_k}{\left| [g_\parallel]_i \right|} [g_\parallel]_i & \text{otherwise,}
\end{cases}
\end{equation*}
where $\lambda_i=\hat{\lambda}_i+\zeta_k$ for $i=1,\ldots,2l$ and $c\in[-\delta_k,\delta_k]$.
The closed form solution to (\ref{eqn-shape-tr2-sep}) 
is 
$$v_\perp^* = \left\{\begin{array}{ll}
-\frac{1}{\zeta_k^C}g_\perp & \text{if } \zeta_k^C>0
\text{ and } \|g_\perp\|_2\le \delta_k |\zeta_k^C|\\
\delta_k u & \text{if }
\zeta_k^C\le 0 \text{ and } 
\|g_\perp\|_2=0\\
-\frac{\delta_k}{\|g_\perp\|_2}g_\perp & \text{otherwise,}\end{array}
\right.$$
where $u\in\Re^{n-2l}$ is a
unit vector with respect to the
two-norm~\cite{burdakov2017efficiently,Oleg2021}.
Notice that $\|v_\perp\|_2$ is at times inversely-proportional
to $\zeta_k^C$ when $\zeta_k^C$ is positive; in other words, a very large and positive $\zeta_k^C$ results in a small $\beta = 1 / \zeta_k $ when $\|g_\perp\|_2$ is not too relatively large.

\medskip

Having obtained optimal $v_\parallel$ and $v_\perp$, $s^*$ can be recovered using the relationship $Pv=s$ and noting
that $P_\perp P_\perp^T = (I-P_\parallel P_\parallel)$
to give
\begin{eqnarray}\label{eqn-sstar}
s^*& = & P[v_\parallel^*+v_\perp^*]\nonumber \\
& = & 
P_\parallel v_\parallel^*+
P_\perp v_\perp.
\end{eqnarray}
To compute $P_\perp v_\perp$,
we use the same strategy as in~\cite{Oleg2021}; that is,
picking $u$ to be 
$u=\frac{P_\perp^Te_i}{\|P_\perp^Te_i\|_2},$ where $i$ is the first
index such that $\|P_\perp^Te_i\|\ne 0$, then 
\begin{equation}\label{eqn-sstar2}
s^* = P_\parallel(v_\parallel^*-
P_\parallel^T w^*) + w^*,
\end{equation}
where $$w^* =
\left\{\begin{array}{ll}
-\frac{1}{\zeta_k^C}g & \text{if } \zeta_k^C>0
\text{ and } \|g_\perp\|_2\le \delta_k |\zeta_k^C|\\ [0.2cm]
\frac{\delta_k}{
\|P_\perp^Te_i\|_2}e_i & \text{if }
\zeta_k^C\le 0 \text{ and } 
\|g_\perp\|_2=0\\[.2cm]
-\frac{\delta_k}{\|g_\perp\|_2}g & \text{otherwise.}\end{array}
\right.$$

Note that the quantities $\|g_\perp\|_2$ and $\|P_\perp^Te_i\|_2$
can be computed using following
relationships:
$$\|g_\perp\|_2^2
+\|g_\parallel\|_2^2=\|g\|_2^2 \quad\text{ and } \quad \|P_\perp^Te_i\|_2^2+
\|P_\parallel^Te_i\|_2^2=1.$$
Thus, the solution $s^*$ for the $(P,\infty)$ trust-region subproblem can be computed using only $P_\parallel$ via  (\ref{eqn-Parallel}).

\subsubsection{The $(P,2)$ shape-changing norm}
In this section, we consider a trust-region subproblem whose constraint is defined by the $(P,2)$ norm:
\begin{equation}\label{eqn-shape-tr2}
\min_{s\in\Re^n} Q(s)=g_k^Ts+\frac{1}{2}s^TB_ks
\quad \text{subject to } \quad \|s\|_{P,2}\le \delta,
\end{equation}
Different from the $(P,\infty)$-norm, the subproblem does not have a closed-form solution; however, it can be decoupled into two subproblems--one that has a closed-form solution and one that is a low-dimensional two-norm subproblem that is easily solved.  To see this, consider the same approach as in the $(P,\infty)$-norm case.  Applying the same change of variables $v=P^Ts$, yields (\ref{eqn-cov1}) as before.  The problem is separable and decouples into the following trust-region subproblems:

\begin{eqnarray}
\min_{v_\parallel\in\Re^{2l}} q_\parallel(v_\parallel)=g_\parallel^Tv_\parallel+\frac{1}{2}\left(v_\parallel^T\left(\hat{\Lambda}+\zeta_kI_{2l}\right)v_\parallel\right)\quad \text{subject to } \quad \|v_\parallel\|_{2}\le \delta, \label{eqn-shape-tr1-sep2}\\
\min_{v_\perp\in\Re^{n-2l}}q_\perp(v_\perp)=g_\perp^Tv_\perp+\frac{1}{2}\zeta^C_k\|v_\perp\|_2^2
\quad \text{subject to } \quad \|v_\perp\|_{2}\le \delta.
\label{eqn-shape-tr2-sep2}
\end{eqnarray}
Since (\ref{eqn-shape-tr2-sep2}) is identical to (\ref{eqn-shape-tr2-sep}), its closed-form solution is given in Section~\ref{sec-Pinfty}.
Subproblem (\ref{eqn-shape-tr1-sep2}) is a low-dimensional problem since $l$ is typically chosen to be a small number (e.g., less than 10).  Moreover, $\nabla^2 q_\parallel (v_\parallel)$ is a diagonal matrix.  For this reason, any standard trust-region method (including direct methods) may be used to solve this subproblem (e.g., see~\cite{conn2000trust} for possible methods).  However, in this work,
we propose using the method found in~\cite{brust2017solving}.

\medskip

The {\small OBS} method found in~\cite{brust2017solving} is an ``exact" subproblem solver when {\small L-SR1} matrices are used as the approximate Hessian.
The method computes solutions to satisfy
optimality conditions given in
Theorem~\ref{thrm-GMS} by exploiting the compact formulation of {\small L-SR1} matrices.  For this work, we use a modified version of the {\small OBS} method that makes use of the compact formulation for {\small MSS} matrices.
Specifically, given an {\small MSS} matrix and its compact formulation
 (Section~\ref{subsec-compact2}),
the partial spectral decomposition
$B_k=P\Lambda P^T$
can be computed as in Section~\ref{eqn-dense-spectral}, where
$$P= \begin{bmatrix} QU & (QU)^\perp\end{bmatrix} \quad \text{and} \quad
\Lambda = \begin{bmatrix} \hat{\Lambda}+\zeta_k I & 0 \\ 0 & \zeta_k^C I \end{bmatrix}.
$$
The optimality conditions given
by Theorem~\ref{thrm-GMS} for the $(P,2)$ subproblem are as follows:
\begin{eqnarray}\label{eqn-parallel-case1}
\left(\hat{\Lambda}+(\sigma^* +\zeta_k ) I\right)v_\parallel^*&=&-g_\parallel,\\
\sigma^*\left(\|v_\parallel^*\|_2-\delta\right) & = & 0\\
\|v_\parallel\|_2 &\le& \delta,\\
\sigma^* & \ge & 0,\\
\hat{\lambda}_i+(\sigma^*+\zeta_k) &\ge& 0 \text{ for } 1\le i \le 2l.
\label{eqn-parallel-case5}\end{eqnarray}
Let $\hat{\lambda}_{2l}$ denote
the smallest entry in $\hat{\Lambda}$.  A solution of (\ref{eqn-parallel-case1})--(\ref{eqn-parallel-case5}) 
can be computed by considering three general cases
that depend
on the sign of $\hat{\lambda}_{2l}+\zeta_k$.
Details for each case
is given in~\cite{brust2017solving,Oleg2021}.

\medskip

Having obtained $v_\perp^*$ and $v_\parallel^*$, the solution to the $(P,2)$-norm shape-changing subproblem is computed using (\ref{eqn-sstar2}).  As with the $(P,\infty)$-norm case, matrix-vector products with $P_\parallel$ do
not require forming $P_\parallel$ explicitly.

%% file: numerical.tex
\section{Numerical results}
 In this section, we report results of various experiments using the limited-memory {\small MSS}  method ({\small L-MSSM}) and other limited-memory quasi-Newton methods.  For these results,
 we used 60 problems from the {\small CUTE}st test set~\cite{cutest} with $n\ge 1000$.  Specifically, all problems with the classification ``OUR2" with $n\ge 1000$ were chosen\footnote{See \url{https://www.cuter.rl.ac.uk/Problems/mastsif.shtml} for further classification information.},
which includes all problems with an objective function that is nonconstant, nonlinear, nonquadratic, and not the sum of squares.  The 60 problems were: {\small ARWHEAD, BOX, BOXPOWER, BROYDN7D, COSINE, CRAGGLVY, CURLY10, CURLY20, CURLY30, DIXMAANA, DIXMAANB, DIXMAANC, DIXMAAND, DIXMAANE, DIXMAANF,
DIXMAANG, DIXMAANH, DIXMAANI, DIXMAANJ, DIXMAANK, DIXMAANL,
DIXMAANM, DIXMAANN, DIXMAANO, DIXMAANP, DQRTIC, EDENSCH, EG2, ENGVAL1, FLETBV3M, FLETCBV2, FLETCBV3, FLETCHBV, FLETCHCR, FMINSRF2, FMINSURF, GENHUMPS, INDEF, INDEFM, JIMACK, NCB20, NCB20B, NONCVXU2,}

\small{ NONCVXUN,
NONDQUAR, POWELLSG, POWER, QUARTC, SCHMVETT, SCOSINE, SCURLY10, SCURLY20, SCURLY30, SENSORS, SINQUAD, SPARSINE, SPARSQUR, SSCOSINE, TOINTGSS,} and {\small VAREIGVL}.

In our comparisons we use the following five algorithms to solve the trust-region subproblems 
with various choices for the approximate Hessian:
\begin{center}
\begin{tabular}{l|l}
Abbreviation & Description\\
\hline
SC-INF 	& the $(P,\infty)$-norm subproblem solver
with $B_0=\gamma_k I$\\
SC-INF-D 	& the $(P,\infty)$-norm subproblem solver with a dense initialization\\
SC-L2 	& the $(P,2)$-norm subproblem solver with
$B_0=\gamma_k I$\\
SC-L2-D 	& the $(P,2)$-norm subproblem solver with a dense initialization
\\ 
trCG 		& truncated CG \cite[Algorithm 7.5.1]{conn2000trust} \\
\end{tabular}
\end{center}
For these experiments, tr{\small CG} was implemented in {\small MATLAB} by the authors.
The {\small L-MSS} shape-changing subproblem solvers were implemented in an algorithm similar to
\cite[Algorithm 5]{Oleg2021}. 
A feature
of this algorithm is that the {\small L-MSS} matrix is updated by every pair $\{({s}_i,{y}_i)\}^k_{{i=k-m+1}} $
as long as the $ s_i $ are linearly independent
(updates are skipped if this condition is not met).
In the case of a {\small MSS} {\small L2} method, previous
numerical results found that $m=3$ outperformed larger memory choices of $m=5$ and $m=7$~\cite{mss-erway,brust2017solving}. Our own experiments for this paper confirmed these results.  For this reason, $m=3$ is used for all experiments with {\small MSS} matrices.

Comparisons on the test set are made using
extended performance profiles as in \cite{MahajanLeyfferKirches11}. These profiles are an extension
of the well-known profiles of Dolan and Mor\'{e} \cite{DolanMore02}. 
 We compare total computational time (and function calls) for each solver on the test set of problems.
 The performance metric $ \rho_s(\tau) $ with a given number of test problems $ n_p $ is
\begin{equation*}
		\rho_s(\tau) = \text{card}\left\{ p : \pi_{p,s} \le \tau \right\} \big/ n_p \quad \text{and} \quad \pi_{p,s} = t_{p,s}\big/ \underset{1\le i \le S, i \ne s}{\text{ min } t_{p,i}},
\end{equation*} 
where $ t_{p,s}$ is the ``output'' (i.e., time) of
``solver'' $s$ on problem $p$. Here $ S $ denotes the total number of solvers for a given comparison. This metric measures
the proportion of how close a given solver is to the best result. The extended performance profiles are the same as the classical ones for $ \tau \ge 1 $.
(In the profiles we include a dashed vertical grey line, to indicate $ \tau = 1 $.) 
The solvers are compared on 
60  large-scale {\small CUTE}st problems. We consider \eqref{eqn-min} to be solved when $\| \nabla f(x_k) \|_\infty < \varepsilon$ with $\varepsilon=5.0\times 10^{-4}$.  In all the performance profiles, $\rho_s(\tau)<1$, which
indicates that no solver was able to solve all the problems in the test set; however,
the value of $\rho_s(\tau)$ at $\tau=32$ indicates the percentage of problems solved in the test set.

\medskip
For the single-parameter initialization, we use the representation
$B_0=\gamma_k I$, 
where 
\begin{equation}\label{eqn-single}
\gamma_k=\underset{k-1\le i \le k-q}{\text{max}}\left\{\frac{y_i^Ty_i}{y_i^Ts_i}\right\},
\end{equation}
and $q$ is the number of stored iterates used to compute $\gamma_k$. This initialization is based on work in~\cite{Oleg2021} that showed that this initialization and the value of $q=5$ works well for 
single-parameter initializations using {\small SC-INF},
{\small SC-L2}, {\small L2}, and tr{\small CG} in the case of limited-memory Symmetric Rank-1 ({\small L-SR1}) Hessian approximations.
For the dense initialization,  $$\zeta_k =\underset{k-1\le i \le k-q}{\text{max}}\left\{\frac{y_i^Ty_i}{y_i^Ts_i}\right\} \quad \text{and} \quad \zeta_k^C=\frac{y_k^Ty_k}{y_k^Ts_k},$$
with $q=5$, i.e., $\zeta_k$ was chosen to be
the single-parameter initialization and $\zeta_k^C$
was chosen to be the well-known initialization for
quasi-Newton methods~\cite{bb88}.
In~\cite{mss-erway}, a variation of this initialization was found to outperform other 
choices for 
$\zeta_k$ and $\zeta_k^C$ in the case of {\small MSS} matrices.  Note that $q$ is the number of 
of stored updates to form the initialization parameters
and not the ``memory'' given by $m$, i.e., the number of stored updates to form $B_k$.

\subsection{Initialization experiments}
In this section, we compare the performance of the dense initialization with the single parameter initialization using
the two shape-changing norms for the trust-region subproblems.  In this set of experiments, only {\small L-MSSM} 
matrices were used to approximate the Hessian, and a maximum of $5,000$ iterations were allowed. 
We begin with the following two experiments:
\begin{itemize}
\item Experiment I.A: Comparison between the single parameter initialization 
({\small L-MSSM-SC-L2})
with the dense initialization
({\small L-MSSM-SC-L2-D})
using the shape-changing $(P,2)$ norm.
\item 
Experiment I.B: Comparison between the single parameter 
initialization
({\small L-MSSM-SC-INF})
with the dense initialization
({\small L-MSSM-SC-INF-D})
using the shape-changing $(P,\infty)$ norm.
\end{itemize}
Fig.~\ref{fig:comp_init_2}
shows that {\small L-MSSM-SC-L2} outperforms 
{\small L-MSSM-SC-L2-D} in computational time,
while Fig.~\ref{fig:comp_init_infty}
shows that {\small L-MSSM-SC-INF-D} outperforms 
{\small L-MSSM-SC-INF} in both computational time and function calls.  These results lead us to the following third experiment, which compares the better method from the previous two experiments:
\begin{itemize}
\item Experiment I.C: Comparison between the single parameter 
initialization solver using the shape-changing $(P,2)$ norm
({\small L-MSSM-SC-L2})
with the dense initialization solver using the shape-changing
$(P,\infty)$ norm
({\small L-MSSM-SC-INF-D}).
\end{itemize}
Figure~\ref{fig:2_infty_compare} reports the results
of Experiment I.C, where
{\small L-MSSM-SC-INF-D}
appears to do better on the test set in terms of time and function evaluations.  Moreover, this method solves more problems on the test set than 
{\small L-MSSM-SC-L2}.

\begin{figure}[t]
    \begin{minipage}{0.48\textwidth}
		\includegraphics[trim=0 0 20 0,clip,width=\textwidth]{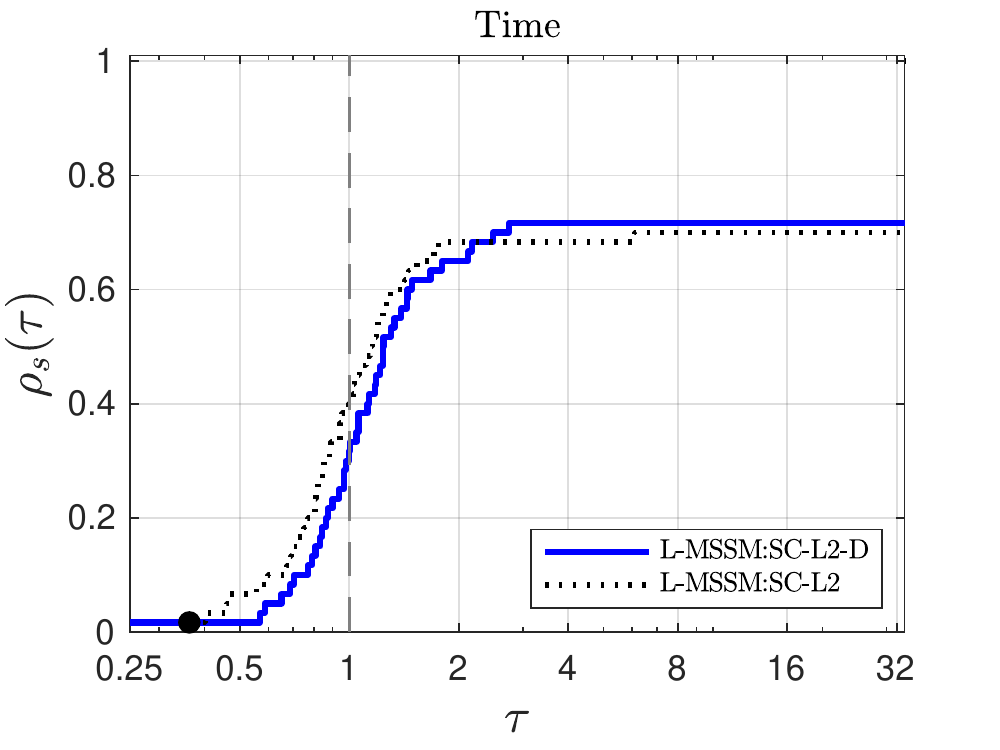}
	\end{minipage}
		\hfill
	\begin{minipage}{0.48\textwidth}
		\includegraphics[trim=0 0 20 0,clip,width=\textwidth]{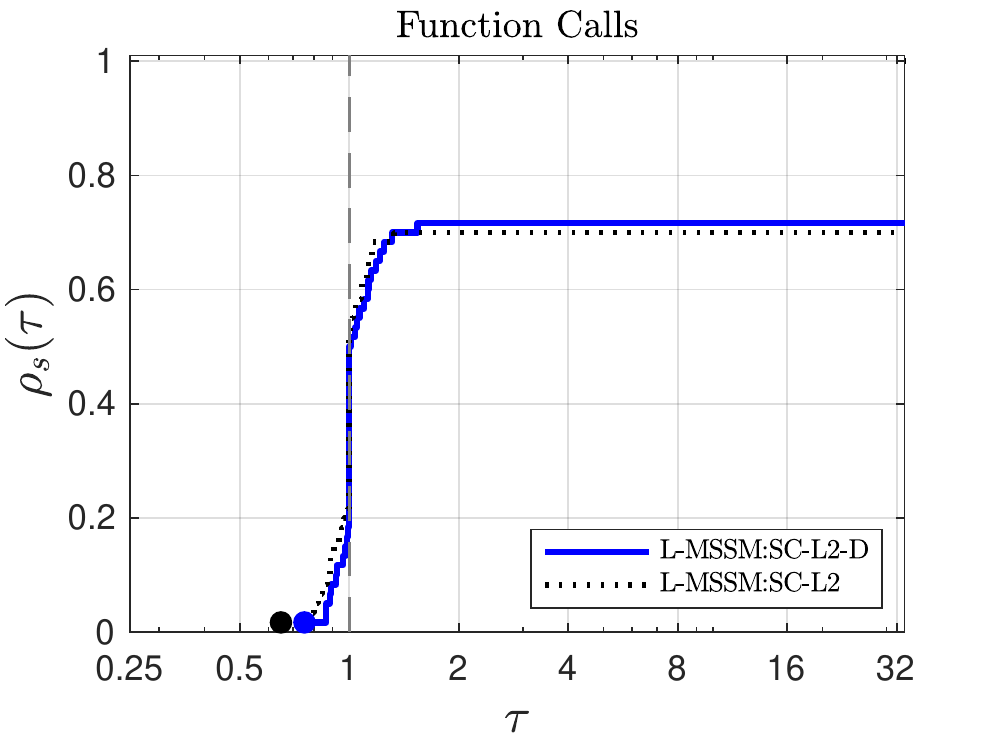}
	\end{minipage}
    \caption{Experiment I.A. Comparison on time and function calls with the shape-changing $(P,2)$ solver.}
    \label{fig:comp_init_2}
\end{figure}
\begin{figure}[t]
    \begin{minipage}{0.48\textwidth}
		\includegraphics[trim=0 0 20 0,clip,width=\textwidth]{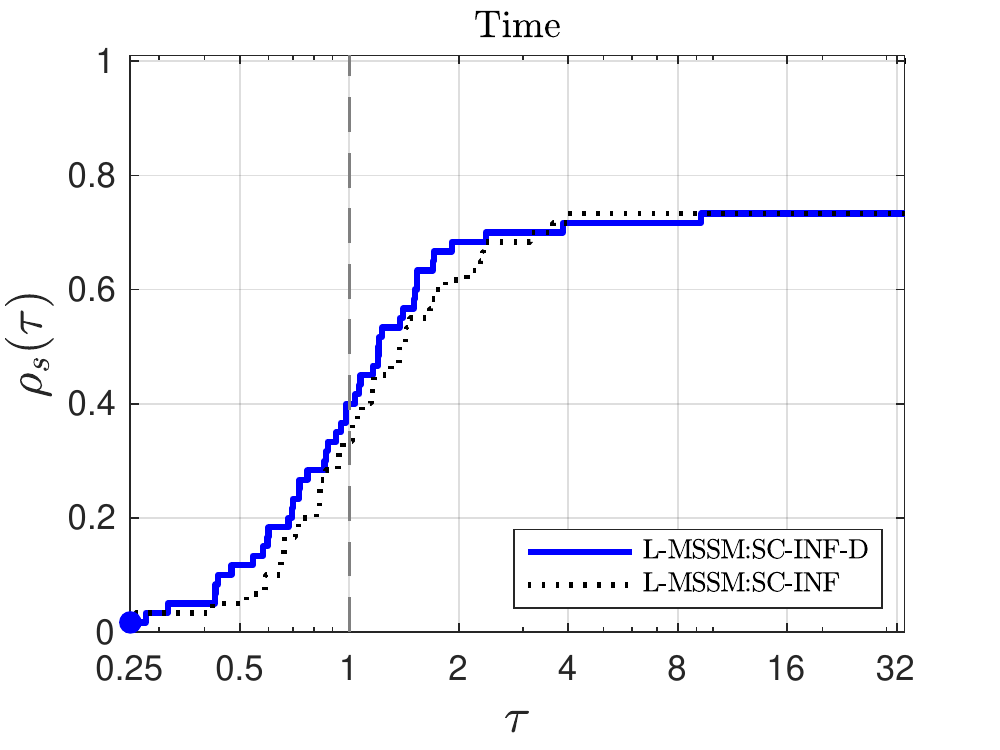}
	\end{minipage}
		\hfill
	\begin{minipage}{0.48\textwidth}
		\includegraphics[trim=0 0 20 0,clip,width=\textwidth]{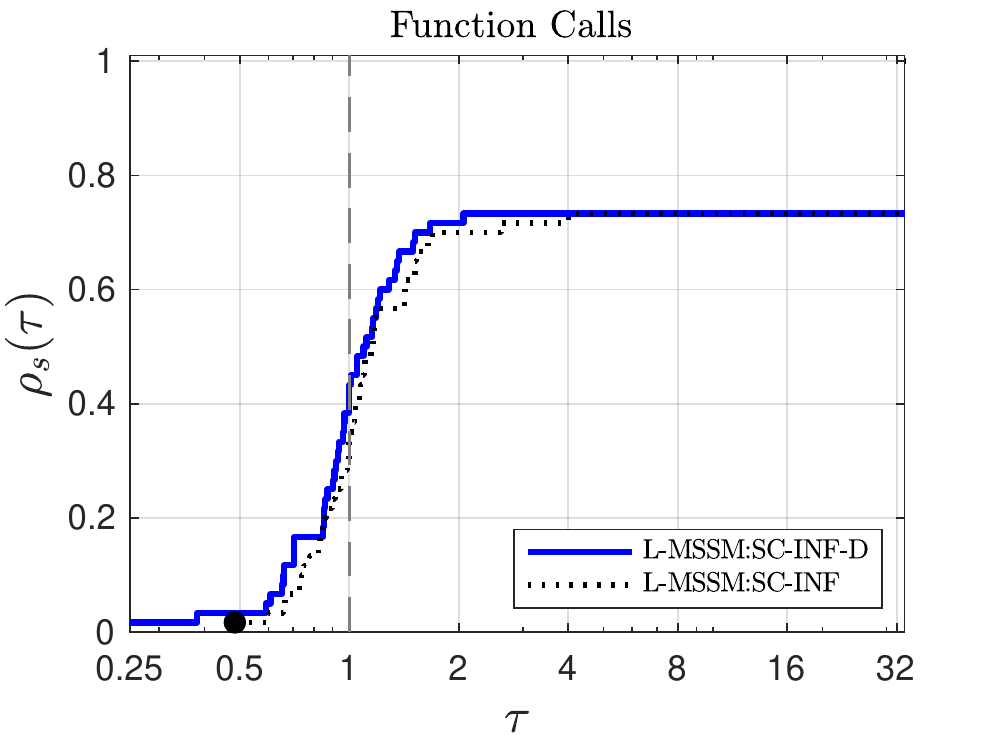}
	\end{minipage}
    \caption{{Experiment I.B. Comparison on time and function calls with the shape-changing $(P,\infty)$ solver.}}
    \label{fig:comp_init_infty}
\end{figure}

\begin{figure}[t]
    \begin{minipage}{0.48\textwidth}
		\includegraphics[trim=0 0 20 0,clip,width=\textwidth]{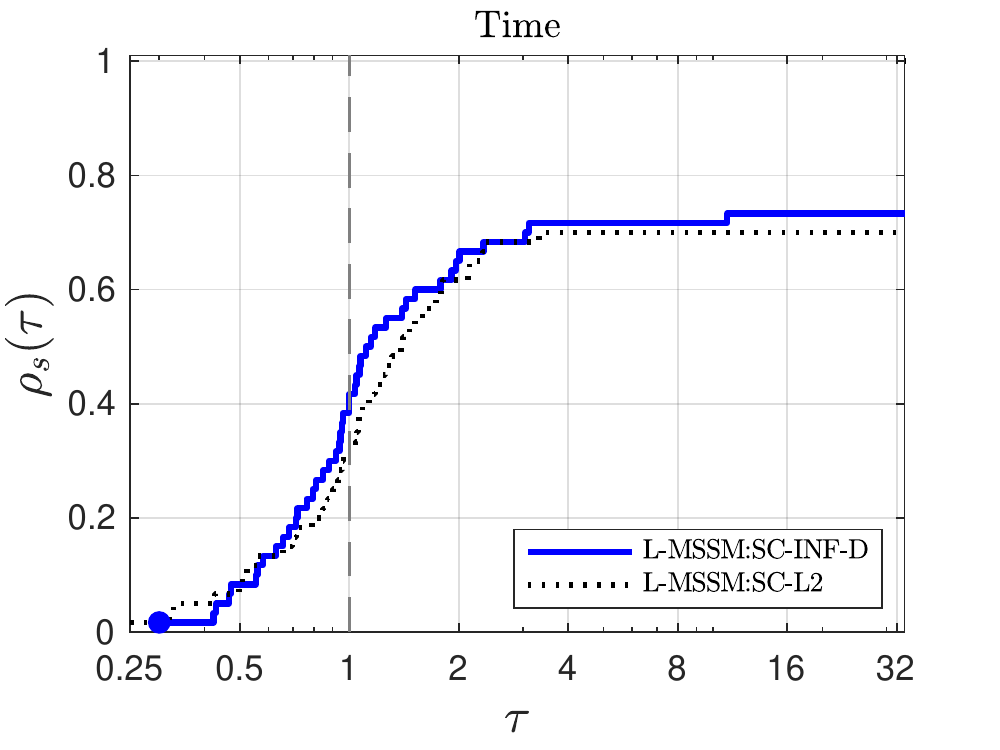}
	\end{minipage}
		\hfill
	\begin{minipage}{0.48\textwidth}
		\includegraphics[trim=0 0 20 0,clip,width=\textwidth]{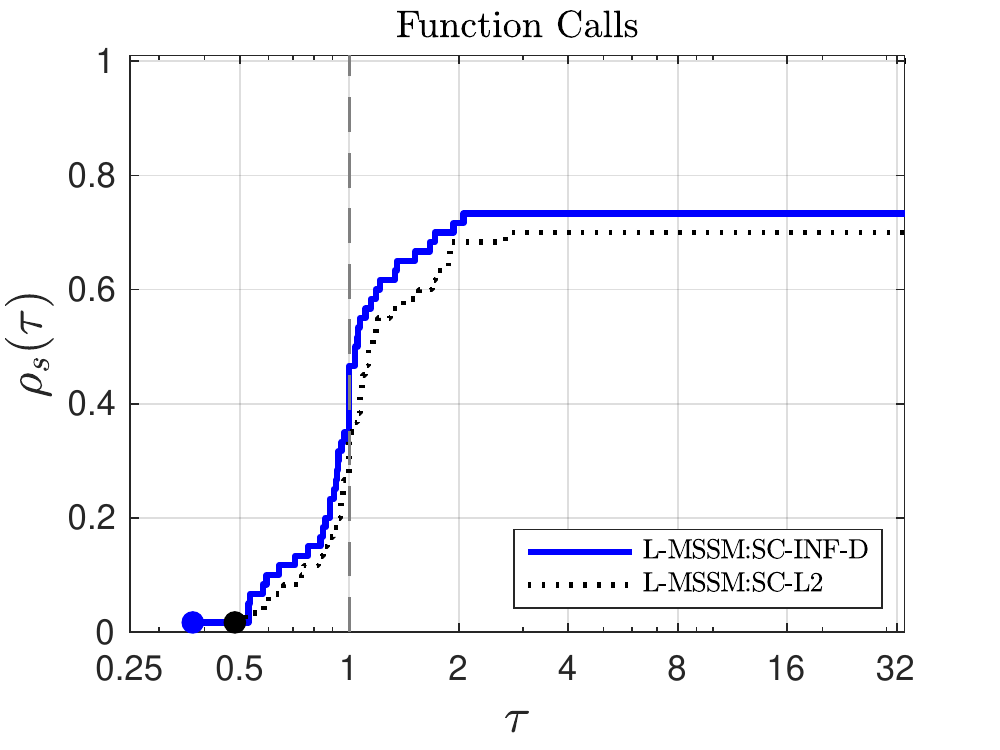}
	\end{minipage}
    \caption{{Experiment I.C. Comparison on time and function calls with the the shape-changing $(P,\infty)$ solver with the dense initialization and the $(P,2)$ solver with the single-parameter initialization.}}
    \label{fig:2_infty_compare}
\end{figure}

\subsection{Subproblem solvers}
In this section, we present Experiment II, which compares  
the best performing solver from Experiment~I 
({\small L-MSSM-SC-INF-D}) to
truncated CG (tr{\small CG}).  
To compare these approaches as trust-region subproblem solvers,
only {\small L-MSSM} approximations of the Hessian are used to approximate the Hessian for both methods.
The maximum number of allowed iterations in this experiment was $5,000$.  The results of this experiment are presented in Figure~\ref{fig:comp_subprob}.
In terms of time, 
{\small L-MSSM-SC-INF-D}
outperforms tr{\small CG}. This may be due to the fact that {\small CG} is an iterative method; in contrast, the $(P,\infty)$-norm solver  analytically computes the solution.
In terms of function evaluations, at $\tau=1$, the $(P,\infty)$-solver outperforms tr{\small CG}--indicating that on any given problem in the subset, the $(P,\infty)$-solver will require fewer function evaluations most of the time.  However, over the entire test set, tr{\small CG} performs slightly better in terms of function evaluations.

\begin{figure}[t]
    \begin{minipage}{0.48\textwidth}
		\includegraphics[trim=0 0 20 0,clip,width=\textwidth]{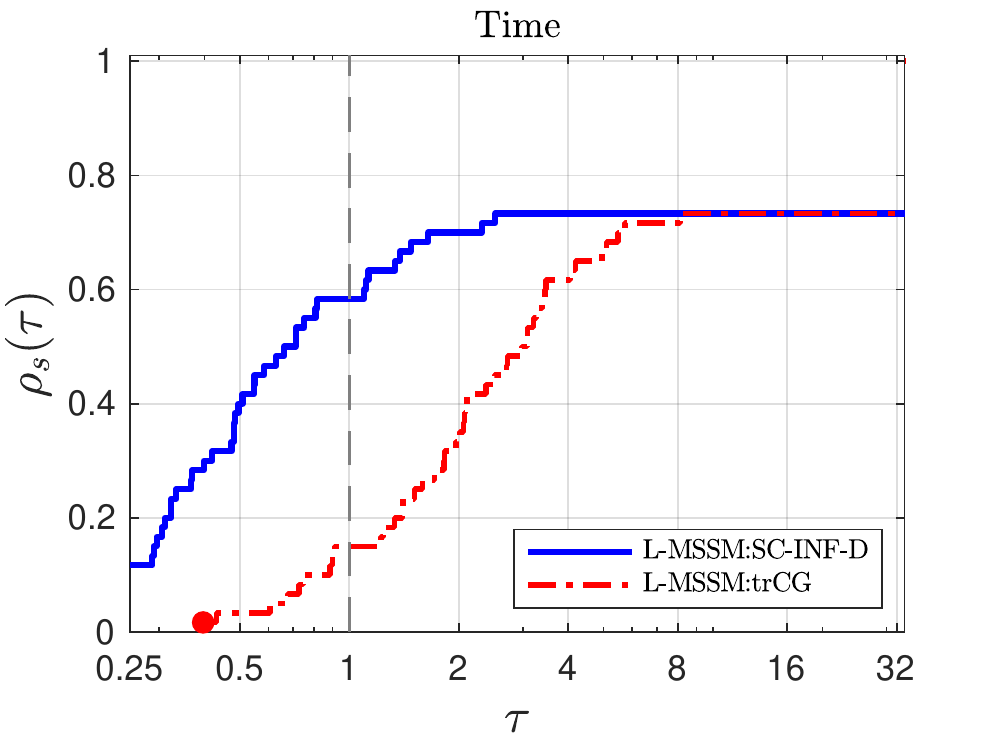}
	\end{minipage}
		\hfill
	\begin{minipage}{0.48\textwidth}
		\includegraphics[trim=0 0 20 0,clip,width=\textwidth]{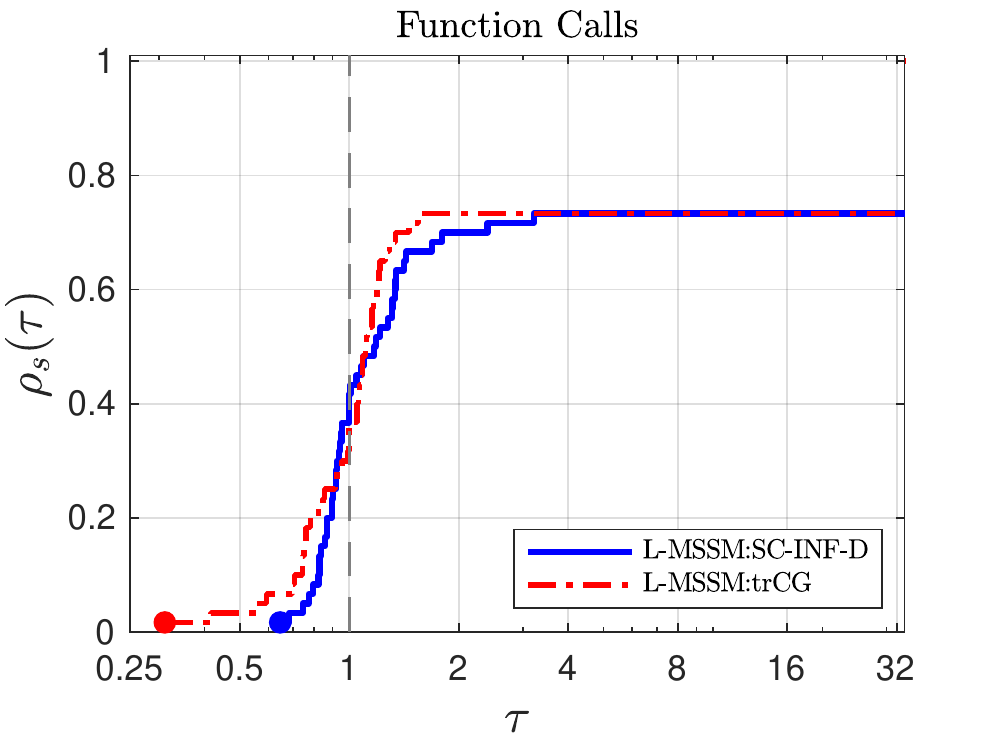}
	\end{minipage}
    \caption{Experiment II. Comparison of $(P,\infty)$ with the dense initialization to trCG subproblem solvers with L-MSSM matrices.}
    
    \label{fig:comp_subprob}
\end{figure}

\subsection{L-SR1 comparison}
In this section, we compare the performance of the shape-changing norms using {\small L-SR1} and {\small L-MSSM} approximations for the Hessian.  
For these experiments, the maximum number of iterations is $50,000$. 
The memory parameters for {\small L-SR1}
were chosen based on results in~\cite{Oleg2021}, where 
$m=5$ (the quasi-Newton method memory parameter) and $q=7$ (the number of stored iterates used to compute $B_0 = \gamma_k I$) appear  to be the
best combination.  
We present four experiments.

\begin{itemize}
\item Experiment III.A: Comparison between solvers using {\small L-SR1 } matrices with the single parameter initialization with 
{\small L-MSSM} matrices with the single parameter initialization using the shape-changing $(P,2)$ norm.
\item Experiment III.B: Comparison between solvers using {\small L-SR1 } matrices with the single parameter initialization with 
{\small L-MSSM} matrices with the dense parameter initialization using the shape-changing $(P,2)$ norm.
\item Experiment III.C: Comparison between solvers using {\small L-SR1 } matrices with the single parameter initialization with 
{\small L-MSSM} matrices with the single parameter initialization using the shape-changing $(P,\infty)$ norm.
\item Experiment III.D: Comparison between solvers using {\small L-SR1 } matrices with the single parameter initialization with 
{\small L-MSSM} matrices with the dense parameter initialization using the shape-changing $(P,\infty)$ norm.
\end{itemize}
In these experiments, we used the single parameter initialization for {\small L-SR1}.  For {\small L-MSSM}, we used both the optimal memory sizes of $m=3$ and $q=5$ as well as the same memory size used for {\small L-SR1} ($m=5$ and $q=7$). 
Figs.~\ref{fig:sr1-p2}-\ref{fig:sr1-pinfty-D} report the results for these four experiments.  In all cases, the solvers that use the {\small L-MSSM} matrices outperform those that use the {\small L-SR1} matrices, both in computational time and function evaluations.  In particular, based on the results of Figure~\ref{fig:2_infty_compare}, it is not surprising that the performance profile comparing {\small L-SR1} with the dense initialization {\small L-MSSM} (Fig.\ \ref{fig:sr1-pinfty-D}) is more striking.


\begin{figure}[t]
     \begin{minipage}{0.48\textwidth}
		\includegraphics[trim=0 0 20 0,clip,width=\textwidth]{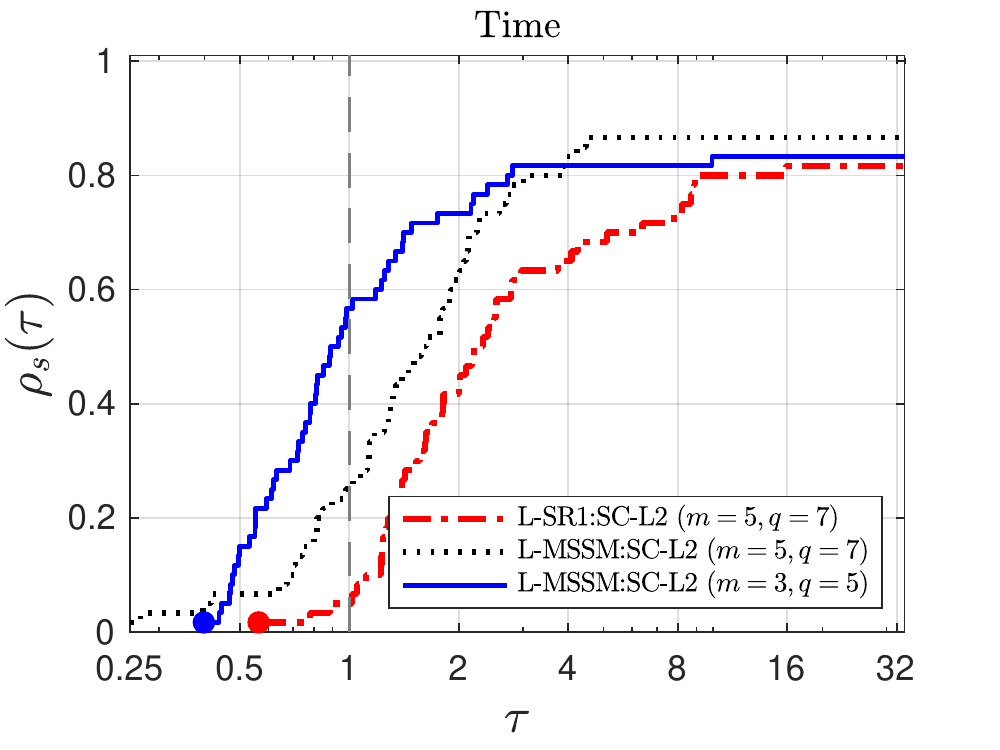}
	\end{minipage}
		\hfill
	\begin{minipage}{0.48\textwidth}
		\includegraphics[trim=0 0 20 0,clip,width=\textwidth]{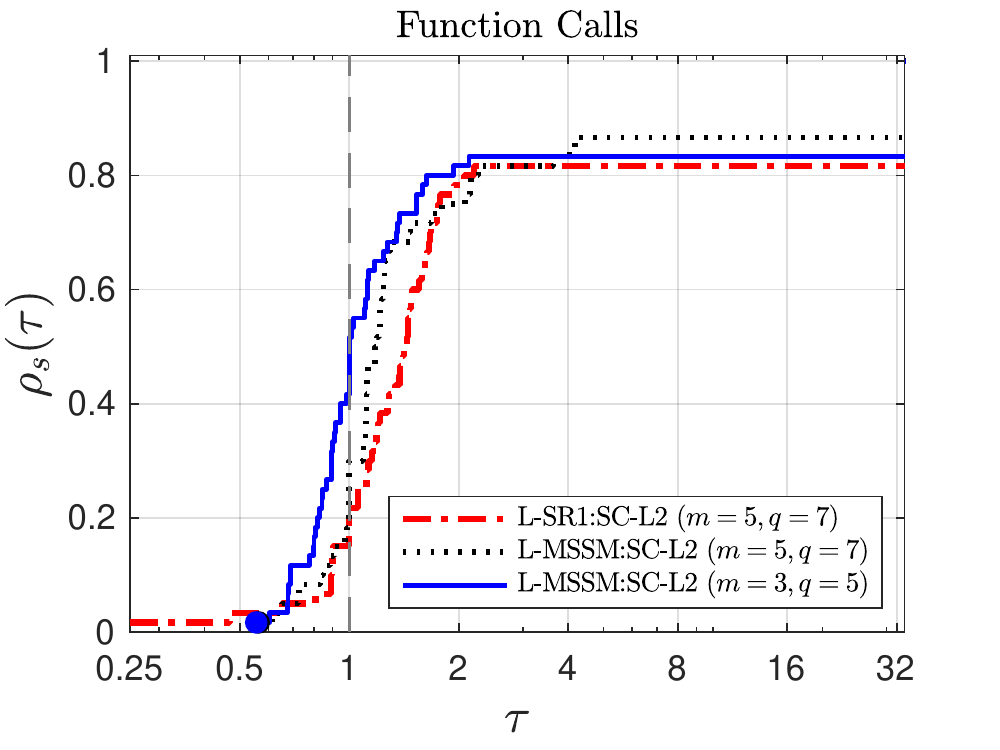}
	\end{minipage}
    \caption{Experiment III.A: Comparison between L-SR1 and L-MSSM matrices using the single-parameter initialization with the shape-changing $(P,2)$ norm.}
    
    \label{fig:sr1-p2}
\end{figure}

\begin{figure}[t]
     \begin{minipage}{0.48\textwidth}
		\includegraphics[trim=0 0 20 0,clip,width=\textwidth]{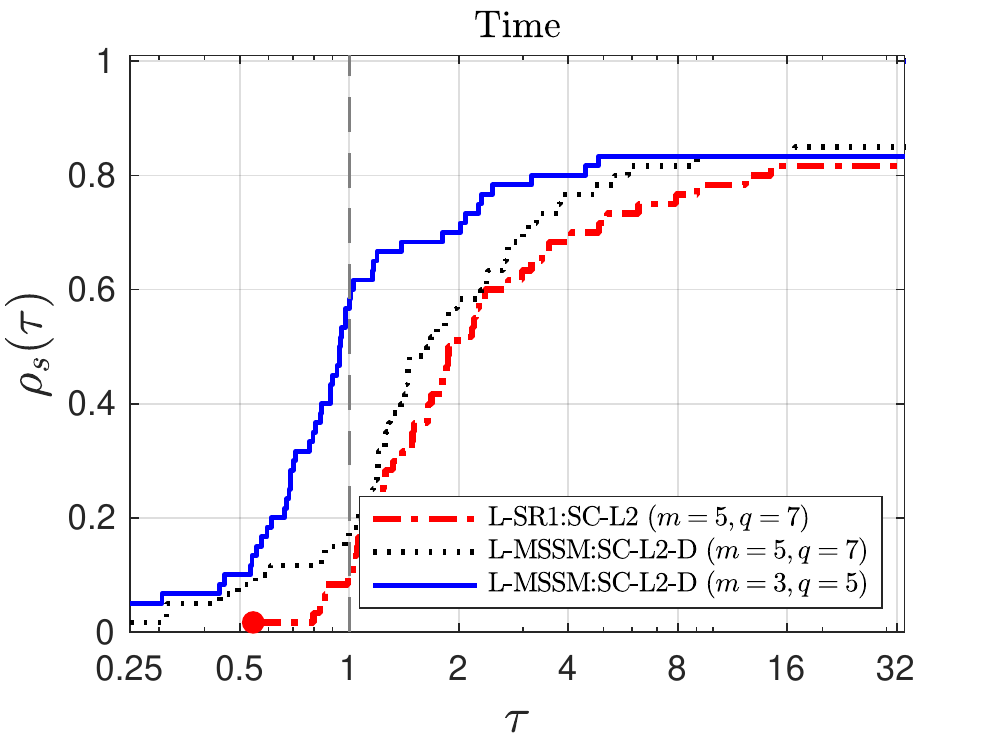}
	\end{minipage}
		\hfill
	\begin{minipage}{0.48\textwidth}
		\includegraphics[trim=0 0 20 0,clip,width=\textwidth]{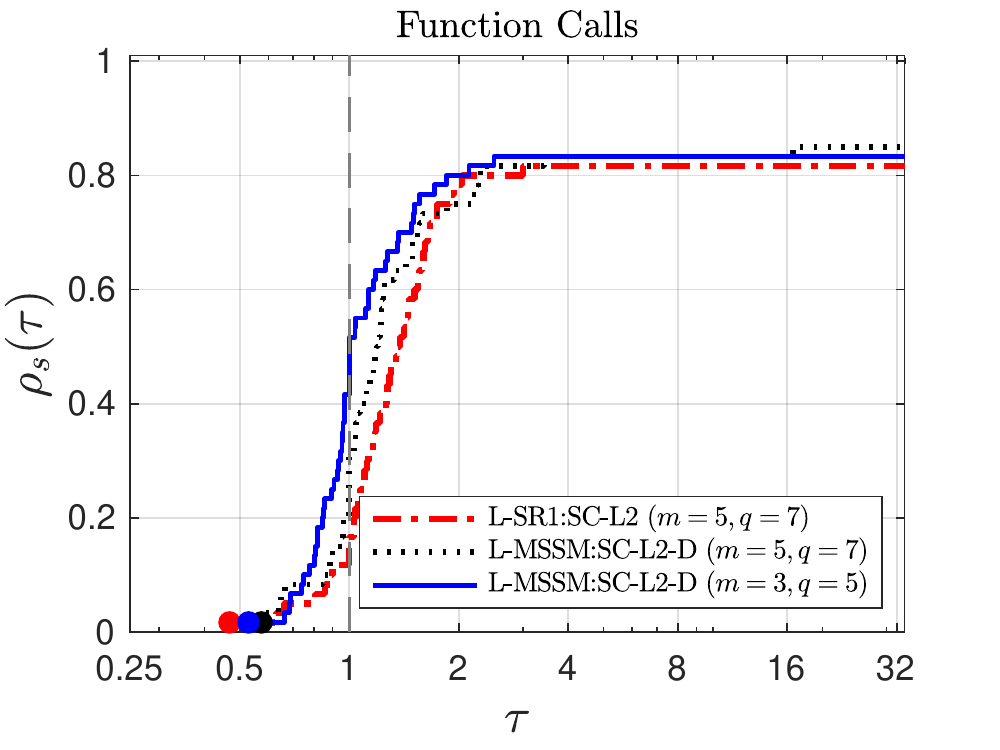}
	\end{minipage}
    \caption{Experiment III.B: Comparison between L-SR1 and L-MSSM matrices using the dense initialization with the shape-changing $(P,2)$ norm.}
    
    \label{fig:sr1-p2-D}
\end{figure}


\begin{figure}[t]
     \begin{minipage}{0.48\textwidth}
		\includegraphics[trim=0 0 20 0,clip,width=\textwidth]{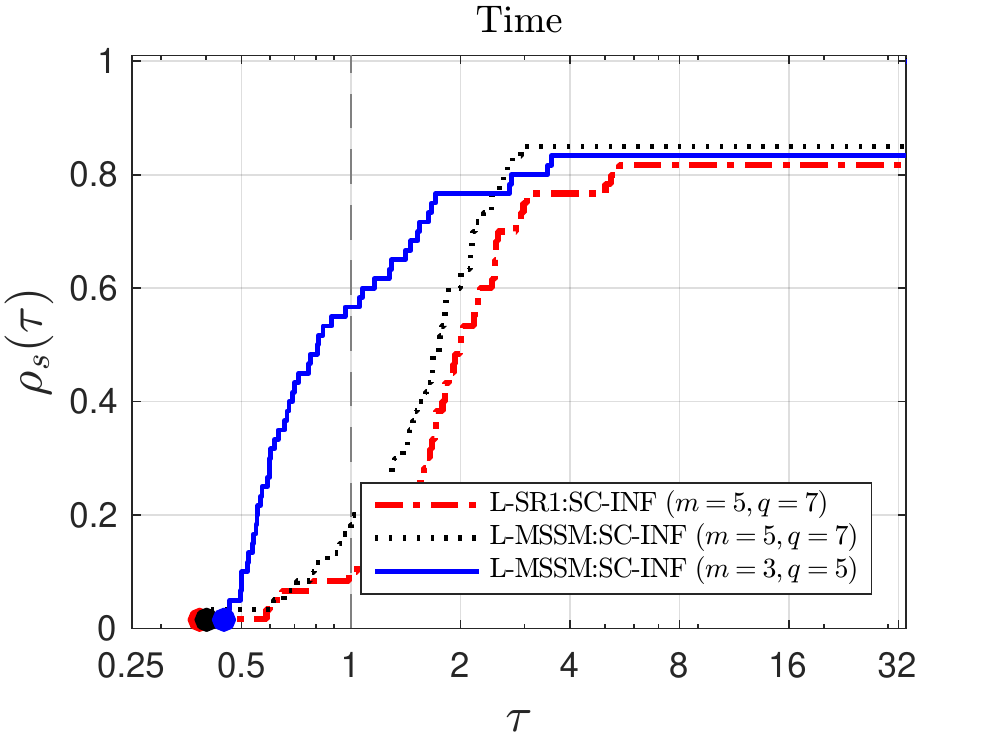}
	\end{minipage}
		\hfill
	\begin{minipage}{0.48\textwidth}
		\includegraphics[trim=0 0 20 0,clip,width=\textwidth]{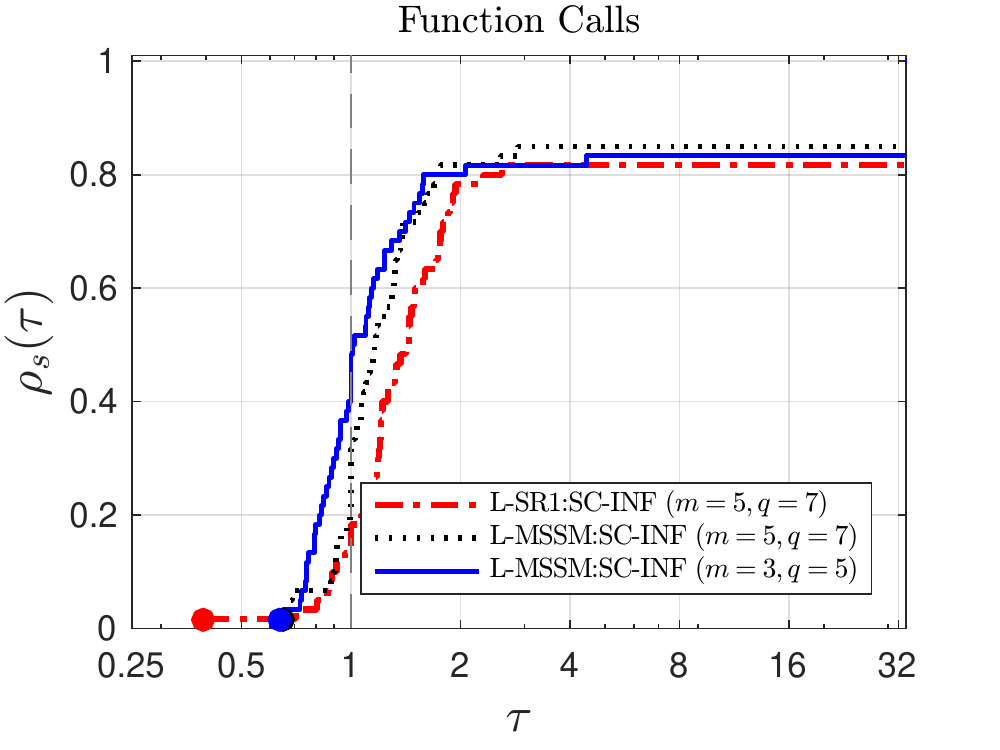}
	\end{minipage}
    \caption{Experiment III.C: Comparison between L-SR1 and L-MSSM matrices using the single-parameter initialization with the shape-changing $(P,\infty)$ norm.}
    
    \label{fig:sr1-pinfty}
\end{figure}

\begin{figure}[t]
     \begin{minipage}{0.48\textwidth}
		\includegraphics[trim=0 0 20 0,clip,width=\textwidth]{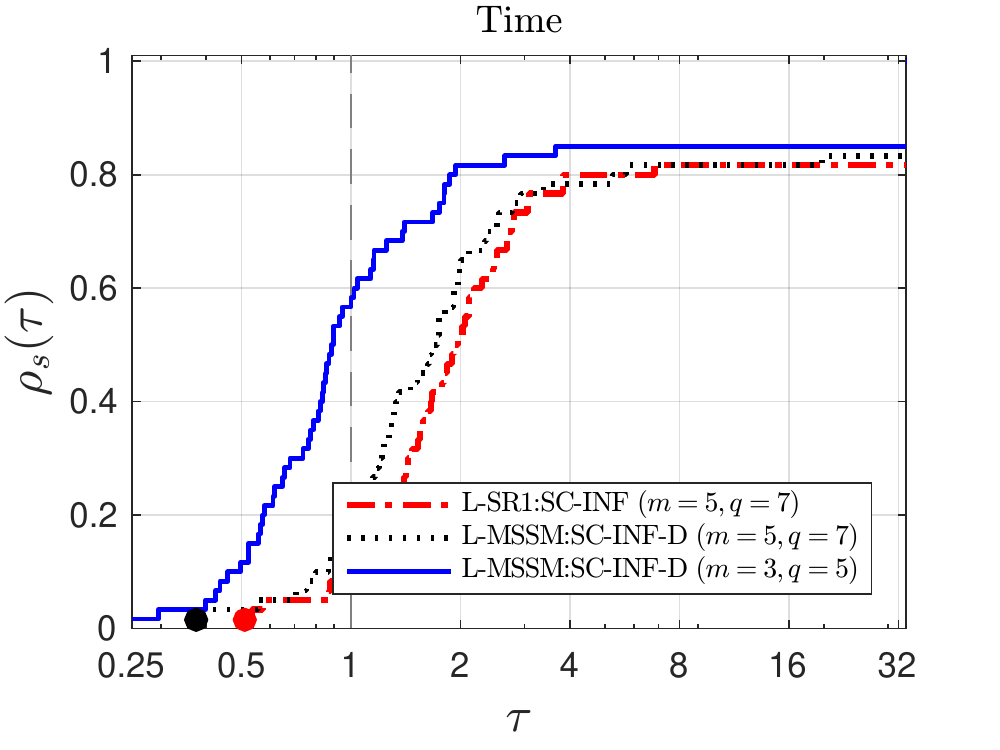}
	\end{minipage}
		\hfill
	\begin{minipage}{0.48\textwidth}
		\includegraphics[trim=0 0 20 0,clip,width=\textwidth]{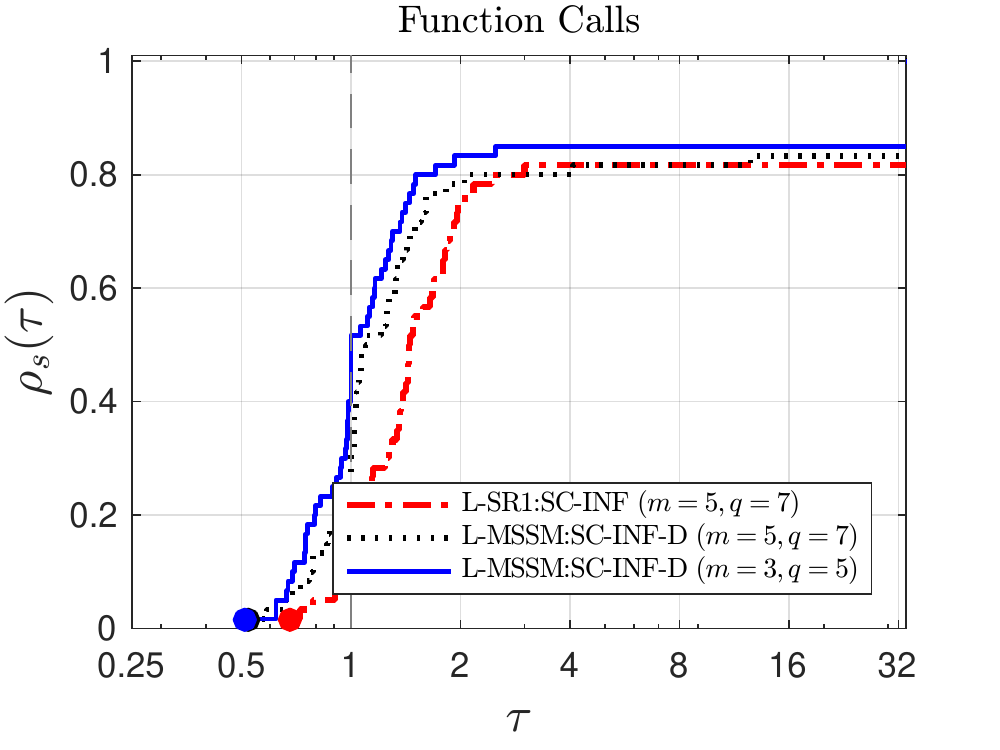}
	\end{minipage}
    \caption{Experiment III.D: Comparison between L-SR1 and L-MSSM matrices using the dense initialization with the shape-changing $(P,\infty)$ norm.}
    \label{fig:sr1-pinfty-D}
\end{figure}

%% file: conclusion.tex
\section{Concluding remarks}
In this paper, we proposed {\small L-MSS} methods
that make use of the dense initialization and two shape-changing
norms.
Numerical results suggest that methods using densely-initialized {\small MSS} 
matrix approximations of the Hessian together with the shape-changing 
norms outperform other trust-region methods. 
Based on the results in this paper,
we suggest default settings of $m=3$ and $q=5$ for
both {\small SC-L2} and {\small SC-INF} when using
either the dense or single-parameter initializations.